\documentclass[12pt]{article}
\usepackage{amsmath}
\usepackage{amsfonts}

\usepackage{graphicx}
\usepackage{color}
\usepackage{epsfig}\usepackage{psfrag}

\newtheorem{theorem}{Theorem}[section]

\newtheorem{corollary}[theorem]{Corollary}

\newtheorem{definition}[theorem]{Definition}
\newtheorem{example}[theorem]{Example}

\newtheorem{lemma}[theorem]{Lemma}

\newenvironment{proof}[1][Proof]{\textbf{#1.} }{\ \rule{0.5em}{0.5em}}
\renewcommand{\text}{\mbox}

\newcommand{\dom}{\mathop{\rm Dom}}

\newcommand{\graph}{\mathop{\rm Graph}}
\newcommand{\antigraph}{\mathop{\rm Antigraph}}

\newcommand{\R}{\mathbf{R}}

\newcommand{\Z}{\mathbf{Z}}

\begin{document}
\baselineskip20pt

\author{Najma Ahmad\thanks{%
Department of Mathematics, University of Toronto, Toronto Ontario M5S 2E4 Canada,
\texttt{ahmad@math.toronto.edu};
{\em Current address:
Ernst and Young, Toronto Ontario Canada} \texttt{najma.ahmad@gmail.com}},
Hwa Kil Kim\thanks{%
School of Mathematics, Georgia Institute of Technology, Atlanta GA 30332-10160 USA \texttt{hwakil@math.gatech.edu};
{\em Current address: Courant Institute for the Mathematical Sciences,
New York University, New York, NY 10012 USA} {\tt hwakil@gmail.com}},
and Robert J. McCann\thanks{%
Department of Mathematics, University of Toronto, Toronto Ontario M5S 2E4 Canada, \texttt{%
mccann@math.toronto.edu}}}

\title{Optimal transportation, topology and uniqueness\thanks{
Formerly titled {\em Extremal doubly stochastic measures and optimal transportation}.}
\thanks{It is a pleasure to thank Nassif Ghoussoub and Herbert Kellerer,  who provided early
encouragement in this direction,  and Pierre-Andre Chiappori, Ivar Ekeland, and Lars Nesheim,
whose interest in economic applications fortified our resolve to persist.
We thank Wilfrid Gangbo, Jonathan Korman, and Robert Pego
for fruitful discussions, Nathan Killoran for
useful references, and programs of the Banff International Research Station (2003) and
Mathematical Sciences Research Institute in Berkeley (2005) for stimulating these
developments by bringing us together.
The authors are pleased to acknowledge the support of
Natural Sciences and Engineering Research Council of Canada Grants 217006-03 and -08
and United States National Science Foundation Grant DMS-0354729. \copyright 2009 by the authors.}}
\date{\today}

\begin{abstract}
The Monge-Kantorovich transportation problem involves optimizing
with respect to a given a cost function.
Uniqueness is a fundamental open question
about which little is known when the cost function is smooth and
the landscapes containing the goods to be transported possess (non-trivial) topology.
This question turns out to be closely linked to a delicate problem
(\# 111) of Birkhoff \cite{Birkhoff48}:
give a necessary and sufficient condition
on the support of a joint probability to guarantee extremality among all
measures which share its marginals.  Fifty years of progress on Birkhoff's question
culminate in Hestir and Williams'
necessary condition  which is
nearly sufficient for extremality; we relax their subtle
measurability hypotheses separating necessity from sufficiency slightly,
yet demonstrate by example that to be sufficient certainly requires some measurability.
Their condition amounts to the vanishing
of the measure $\gamma$ outside a countable alternating sequence of graphs and antigraphs
in which no two graphs (or two antigraphs) have domains that overlap, and
where the domain of each graph / antigraph in the sequence contains the range
of the succeeding antigraph (respectively, graph).  Such sequences are called
{\em numbered limb systems}. We then explain how this characterization can be used to
resolve the uniqueness of Kantorovich solutions for optimal transportation on a manifold
with the topology of the sphere.

\end{abstract}

\maketitle

\section{Introduction}

This survey weaves together two themes:  the first is Monge's 1781
problem \cite{Monge81} of transporting mass
from a landscape $X$ to a landscape $Y$ so as to minimize the average cost $c(x,y)$
per unit transported;
the second is Birkhoff's 1948 problem  \cite{Birkhoff48}
of characterizing extremality among doubly stochastic measures on the unit square.

The first problem has become classical in the calculus of variations; it has
deep connections to analysis
\cite{Trudinger94} \cite{McCann94} \cite{McCann97}
\cite{Cordero-ErausquinNazaretVillani04} \cite{FigalliMaggiPratelli10p},
geometry  \cite{OttoVillani00} \cite{CorderoMcCannSchmuckenschlager01} \cite{Sturm06ab}
\cite{LottVillani09} \cite{Lott09}  \cite{Villani09} \cite{McCannTopping10} \cite{KimMcCannWarren09p},
dynamics \cite{AmbrosioGigliSavare05} \cite{BernardBuffoni07} \cite{JordanKinderlehrerOtto98} \cite{Otto01}
and nonlinear partial differential equations
\cite{Brenier91} \cite{Caffarelli92} \cite{Caffarelli96b} \cite{Delanoe91} \cite{EvansGangbo99} \cite{MaTrudingerWang05} \cite{Urbas97},
as well as applications in physics \cite{Dobrushin70} \cite{Tanaka73} \cite{McCann98},
statistics \cite{RachevRuschendorf98},
engineering \cite{BouchitteButtazzoSeppecher97} \cite{BouchitteGangboSeppecher08} \cite{GlimmOliker03} \cite{Plakhov04b} \cite{Wang04},
atmospheric modeling \cite{CullenPurser84}  \cite{PurserCullen87} \cite{CullenPurser89} \cite{Cullen06},
and economics \cite{Carlier01} \cite{CarlierEkeland04} \cite{Ekeland10} \cite{ChiapporiMcCannNesheim10}
\cite{FigalliKimMcCann-econ}.
The second is a problem in functional analysis,  at the junction between
measure theory and convex geometry.  It is not evident that either involves
differential topology.

The two problems are linked by Kantorovich's reformulation of Monge's nonlinear minimization
as an (infinite-dimensional) linear program \cite{Kantorovich42} \cite{Kantorovich48}.
In this framework, existence of solutions became straightforward for any
continuous cost $c\in C(X \times Y)$.
Still, fifty more years would elapse before
the optimal volume-preserving map between two arbitrary domains
sought by Monge was constructed
for the Euclidean distance $c(x,y)=|x-y|$ in \cite{Ambrosio03}
\cite{CaffarelliFeldmanMcCann00} and \cite{TrudingerWang01}.
Evans and Gangbo had already solved the analogous problem with the domains replaced
by disjoint Lipschitz continuous probability densities \cite{EvansGangbo99},
while Sudakov's earlier construction
\cite{Sudakov76} required a claim which turned out to be
true only two in dimensions \cite{Ambrosio03} \cite{BianchiniCavalletti10p};
see \cite{Caravenna09p} \cite{ChampionDePascale09p} \cite{BianchiniCavalletti10p}
for simplifications
and \cite{FeldmanMcCann02r} \cite{AmbrosioKirchheimPratelli04} \cite{BernardBuffoni06}
\cite{FigalliRSMUP07} for extensions.
Uniqueness fails in this context \cite{GangboMcCann96} \cite{FeldmanMcCann02u}.
In the meantime both Monge and Kantorovich problems were found
to enjoy unique solutions
for strictly convex costs such as $c(x,y)=|x-y|^p/p$,
with $p=2$
\cite{Brenier87} \cite{Brenier91} \cite{CuestaMatran89} \cite{Cuesta-AlbertosTuero-Diaz93}
and $p>1$
\cite{Caffarelli96} \cite{GangboMcCann95} \cite{GangboMcCann96}  \cite{Ruschendorf95} \cite{Ruschendorf96}.
 A general criterion for
existence and uniqueness of optimal maps was identified by Gangbo \cite{Gangbo95}
and Levin \cite{Levin99},  building on works of those cited above.
For any pair of destinations $y_1 \ne y_2$ in $Y$,
it prohibits the function
\begin{equation}\label{relative cost}
x \in X \longmapsto c(x,y_1) - c(x,y_2)
\end{equation}
from having critical points on $X$.  Strictly convex functions
$c(x,y)=h(x-y)$ on $X=Y= \R^n$ \cite{Caffarelli96} \cite{GangboMcCann95}
satisfy this condition --- called the twist criterion in \cite{Villani09} ---
but no differentiable cost $c \in C^1(X \times Y)$ satisfies it on any
compact manifold $X$ without boundary.
Although terrestrial transportation takes place on the sphere,
there are few theorems set in topologies other than the ball ---
not to speak of the more exotic landscapes which arise naturally in some applications.
Spherical examples typically show that uniqueness of Kantorovich solutions holds even though Monge
solutions fail to exist \cite{GangboMcCann00}.  Building on these developments,
one of the goals of this article is to expose a criterion for uniqueness of Kantorovich solutions
which works equally well on the sphere and the ball \cite{ChiapporiMcCannNesheim10}.
Called the {\em subtwist} by Chiappori McCann and Nesheim,
this criterion depends on the Morse structure of the cost globally: it
permits the function \eqref{relative cost} to have up to two critical points on $X$
--- a unique global minimum and a unique global maximum.   Unfortunately,
this cannot be satisfied in more exotic topologies such as the $k$-holed torus ($k \ge 1$),
where uniqueness
remains a tantalizing open question. Our discussion is predicated on global differentiability
of the cost, since a wide variety of existence and uniqueness results concerning
optimal solutions to the Monge-Kantorovich problem
have been established for costs with singular sets 
--- including distances in
Riemannian \cite{Cordero-Erausquin99T} \cite{McCann01} \cite{FigalliSIAM07},
sub-Riemannian \cite{AmbrosioRigot04} \cite{AgrachevLee09} \cite{FigalliRifford10}
and Alexandrov \cite{Bertrand08} spaces,
and the mechanical actions arising from Tonelli
Lagrangians \cite{BernardBuffoni06} \cite{FathiFigalli10}.

The proof that the subtwist condition is sufficient for uniqueness relies on progress in Birkhoff's
problem of characterizing extremal doubly stochastic measures on the square.
This problem is esoteric and subtle:  although still not completely resolved,
substantial results have been obtained in the six decades since it was posed
\cite{Douglas64} \cite{Lindenstrauss65} \cite{Losert82} \cite{BenesStepan87} \cite{HestirWilliams95}.
Highlights are surveyed below.

The literature surrounding Birkhoff's problem is modest,  compared to the
recent explosion of research on the
Monge-Kantorovich transportation problem.  We expect the main interest of this
article will therefore lie in its connection to the latter developments.
For simplicity of exposition,  however, we postpone a further description of these connections
to section \ref{S:unique}.  The earlier sections are devoted to
Birkhoff's problem and the issues surrounding it.  Although
less familiar than the Monge-Kantorovich theory to most of our readership,  the
developments surveyed are elementary yet powerful; they
require nothing more sophisticated than measure
theory to discuss.  Readers in need of motivation ---
or those interested primarily in optimal transportation --- are encouraged
to skip directly  to Theorem \ref{T:unique} for a 
preview of the intended application.

\section{Extremal doubly stochastic measures}

An $n \times n$ {\em doubly stochastic matrix} refers to a matrix of non-negative entries
whose columns and rows each sum to $1$.  The doubly stochastic
matrices form a convex subset of all $n \times n$ matrices --- in fact a
convex polytope, whose extreme points are in bijective correspondence with the $n!$
permutations on $n$-letters,  according to Birkhoff \cite{Birkhoff46} and
von Neumann \cite{vonNeumann53}.
For example, the $3 \times 3$ doubly stochastic matrices,
$$
\left(
\begin{matrix}
s & t & 1-s-t \\
u & v & 1-u-v \\
1-s-u & 1-t-v & s+t+u+v -1 \\
\end{matrix}
\right)
$$
form a 4-dimensional polytope with 6 vertices.
Shortly after proving this characterization,
Birkhoff  \cite[Problem 111]{Birkhoff48} initiated the search for a
infinite-dimensional generalization, thus
stimulating a line of research which remains fruitful even today.

A {\em doubly stochastic measure} on the square refers to a non-negative Borel
probability measure on $[0,1]^2$ whose horizontal and vertical marginals both coincide
with Lebesgue measure $\lambda$ on $[0,1]$.  The set of doubly stochastic measures
forms a convex set we denote by $\Gamma(\lambda,\lambda)$
(which is weak-$*$ compact in the Banach space dual to
continuous functions $C([0,1]^2)$ normed by their suprema $\|\cdot\|_\infty$).
A measure is said to be {\em extremal} in
$\Gamma(\lambda,\lambda)$ if it cannot be decomposed as a convex combination
$\gamma = (1-t) \gamma_0 + t \gamma_1$
with $0<t<1$ and $ 0\le \gamma_0 \ne \gamma_1 \in \Gamma(\lambda,\lambda)$.
Since the Krein-Milman theorem asserts that convex
combinations of extreme points are dense (in any compact convex subset of a topological
vector space, Figure \ref{fig.extreme}),  it is natural to want to characterize the extreme points of
$\Gamma(\lambda,\lambda)$.  Another motivation for such a characterization
is that every continuous linear functional on $\Gamma(\lambda,\lambda)$ is
minimized at an extreme point.  Whether or not this extremum is uniquely
attained can be an interesting question, as in the optimal transportation context:
in Figure \ref{fig.extreme} the horizontal coordinate is
minimized at a single point but maximized at two extreme points (and along
the segment joining them).

\begin{figure}[h]
\psfragscanon
\centering
\psfrag{o}{$o$}
\epsfig{file=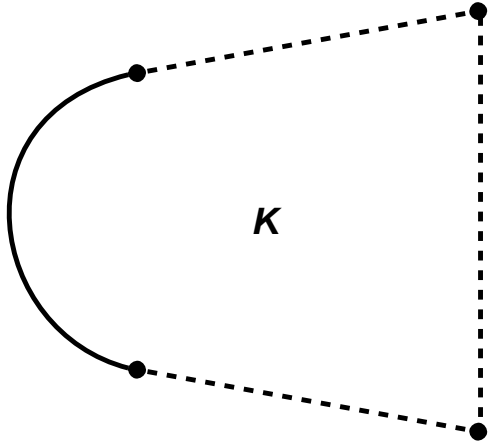, height=6cm}
\caption[Extreme points]
{\label{fig.extreme} Krein-Milman asserts a compact convex set $\mathbf{K}$ can be reconstructed from its extreme
points (denoted here by solid circles $\bullet$ and solid lines $\boldsymbol{-}$).}
\end{figure}

Motivated by the optimization problems already mentioned,
we prefer to formulate the question in slightly greater generality,  by replacing
the two copies of $([0,1],\lambda)$ with probability spaces
$(X,\mu)$ and $(Y,\nu)$,  where $X$ and $Y$ are each subsets of a complete
separable metric space,  and $\mu$ and $\nu$ are Borel probability measures on
$X$ and $Y$ respectively.  This widens applicability of the answer to this question
without increasing its difficulty.
Letting $\Gamma(\mu,\nu)$ denote the Borel probability measures on $X \times Y$
having $\mu$ and $\nu$ for marginals,  we wish to characterize the extreme points of
the convex set $\Gamma(\mu,\nu)$.  Ideally, as in the finite-dimensional case,
this characterization would be given in terms of some geometrical property of the
support of the measure $\gamma$ in $X \times Y$.  Indeed,  if
$\mu = \sum_{i=1}^m m_i \delta_{x_i}$ and $\nu = \sum_{j=1}^n n_j \delta_{y_j}$
are finite, our problem reduces
to characterizing the extreme points of the convex set ${\cal A}$
of $m \times n$ matrices with prescribed column and row sums:
$$
{\cal A} = \{ a_{ij} \ge 0 \mid m_i = \sum_{j=1}^n a_{ij}, \sum_{i=1}^m a_{ij} = n_j \}.
$$
A matrix $(a_{ij})$ is well-known to be extremal in ${\cal A}$ if and only if it is
{\em acyclic},
meaning for every sequence
$a_{i_1 j_1}, \ldots, a_{i_k j_k}$ of non-zero entries
occupying $k \ge 2$ distinct columns and $k$ distinct rows,
the product $a_{i_1 j_2} \ldots a_{i_{k-1} j_{k}} a_{i_k j_1}$ must vanish
--- see Figure \ref{fig.acyclic} or Denny \cite{Denny80}, where the terminology
{\em aperiodic} is used.  Similarly,  a set $S \subset X \times Y$ is
acyclic if for every $k \ge 2$ distinct points $\{x_1, \ldots, x_k\} \subset X$
and $\{y_1,\ldots,y_k\} \subset Y$,  at least one of the pairs
$(x_1 y_1), (x_1,y_2), (x_2,y_2), \ldots, (x_{k-1},y_k),(x_k, y_k),(x_k,y_1)$
lies outside of $S$.

\begin{figure}[h]
\psfragscanon
\centering
\psfrag{o}{$o$}
\epsfig{file=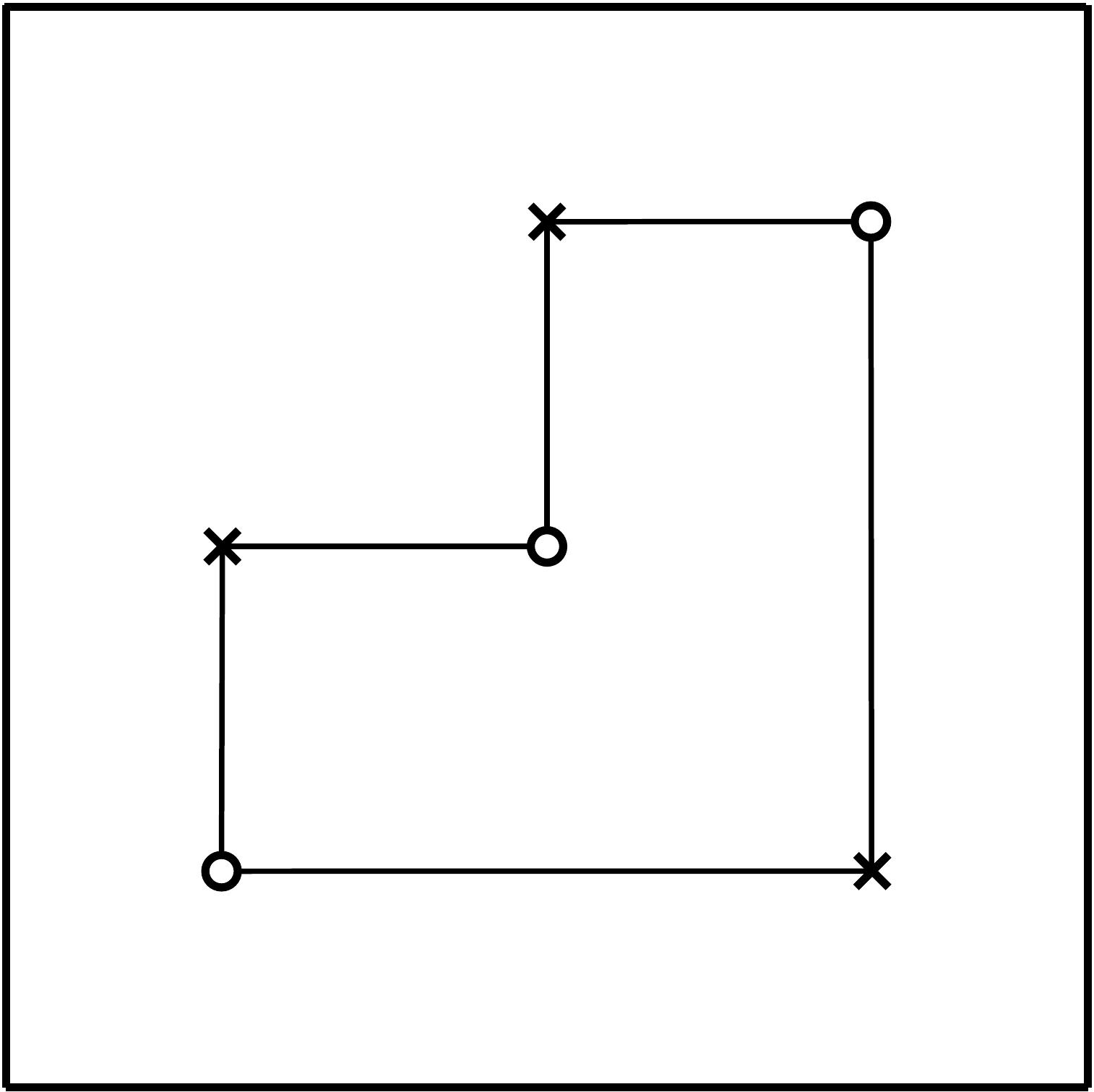, height=6cm}
\caption[Acyclic matrix]
{\label{fig.acyclic} In an {\em acyclic} matrix the product of {\bf x}'s and {\bf o}'s must vanish.}
\end{figure}

A functional analytic characterization of extremality was supplied by
Douglas \cite{Douglas64} and by Lindenstrauss \cite{Lindenstrauss65}:
it asserts that $\gamma$ is extremal in $\Gamma(\mu,\nu)$ if and only if
$L^1(X,d\mu) \oplus L^1(Y,d\nu)$ is dense in $L^1(X \times Y,d\gamma)$.
Although this result is a wonderful starting point,  it is not quite the
characterization we desire for applications,  since it is not easily
expressed in terms of the geometry of the support of $\gamma$.
Significant further progress was made by Bene\v{s}
and \v{S}t\v{e}p\'{a}n, who showed every extremal doubly stochastic
measure vanishes outside some acyclic subset $S \subset X \times Y$ \cite{BenesStepan87}.
Hestir and Williams refined this condition, showing that it
becomes sufficient under an additional Borel measurability hypothesis which,
unfortunately, is not always satisfied \cite{HestirWilliams95}. Some of the subtleties
of the problem were indicated already by Losert's counterexamples \cite%
{Losert82}.  The difficulty of the problem resides partly in
the fact that any geometrical characterization of optimality must be
invariant under arbitrary measure-preserving transformations applied independently
to the horizontal (abscissa) and vertical (ordinate) variables.

In the next two sections we review this line of research, clarifying the nature of
the gap separating necessity from sufficiency and pointing out that
it can be narrowed slightly by replacing the Borel $\sigma$-algebra with suitably
adapted measure-completions. We give a self-contained proof of that part of the
theory which is needed to resolved the uniqueness of optimal transportation with
respect to a smooth cost on the sphere.  This application was first developed in an
economic context by Chiappori, McCann, and Nesheim \cite{ChiapporiMcCannNesheim10},
and forms the subject of the final section of the present manuscript.

\section{Measures on graphs are push-forwards}

Before recalling the characterization of interest,  let us develop a bit of notation
in a simpler setting, and a key argument that we shall require.  Impatient or
knowledgeable readers can 
proceed directly to the final sections below, referring
back to the present section only as needed.

Let $X$ and $Y$ be subsets of complete separable metric spaces,  and fix a
non-negative Borel measure $\mu$ on $X$.
Suppose $f:X \longrightarrow Y$ is {$\mu$-measurable},
meaning $f^{-1}(B)$ is in the $\sigma$-algebra
completion of the Borel subsets of $X$ with respect to the measure $\mu$,
whenever $B$ is relatively Borel in $Y$.
Then a Borel measure on $Y$ is induced,  denoted $f_\#\mu$ and called the
{\em push-forward} of $\mu$ through $f$, and given by
\begin{equation}\label{push-forward}
(f_\#\mu)[B] := \mu[f^{-1}(B)]
\end{equation}
for each Borel $B \subset Y$.  Defining the projections
$\pi^{X_{}}(x,y) = x$ and $\pi^{Y_{}}(x,y)=y$ on $X \times Y$,  this notation permits
the horizontal and vertical marginals of a measure $\gamma \ge 0$ on $X \times Y$
to be expressed as $\pi^{X_{}}_\# \gamma$ and $\pi^{Y_{}}_\# \gamma$ respectively.

The next lemma shows that any measure supported on a graph can be deduced
from its horizontal marginal.  It improves on Lemma 2.4 of \cite{GangboMcCann00}
and various other antecedents,  by using an argument from Villani's
Theorem 5.28 \cite{Villani09}
to extract $\mu$-measurability of $f$ as a conclusion rather that a hypothesis.
As work of, e.g., Hestir and Williams \cite{HestirWilliams95} implies,
although measures on graphs are extremal in $\Gamma (\mu ,\nu )$, the converse is
far from being true; this peculiarity is an inevitable consequence of the infinite
divisibility of $(X,\mu)$.

\begin{lemma}[Measures on graphs are push-forwards]
\label{pure implies unique} Let $X_{}$ and $Y_{}$ be subsets of
complete separable metric spaces, and $\gamma \geq 0$ a $\sigma $-finite
Borel measure on the product space $X_{}\times Y_{}$. Denote the horizontal
marginal of $\gamma $ by $\mu _{}:=\pi _{\#}^{X_{}}\gamma $. If $\gamma $
vanishes outside the graph of $f:X_{}\longrightarrow Y_{}$, meaning
$\{(x,y)\in X_{}\times Y_{}\mid y\neq f(x)\}$ has zero outer measure, then $%
f$ is $\mu _{}$-measurable and $\gamma =(id_{X_{}}\times f)_{\#}\mu _{}$,
where $id_X \times f$ denotes the map $x \in X \longmapsto (x,f(x)) \in X \times Y$.
\end{lemma}

\begin{proof}
Since outer-measure is subadditive, it costs no generality to assume the
subsets $X_{}$ and $Y_{}$ are in fact complete and separable,
by extending $\gamma$ in the obvious (minimal) way. Any $%
\sigma $-finite Borel measure $\gamma$ is regular and $\sigma$-compact on a
complete separable metric space; e.g.\ p.~255 of \cite{Dudley02} or Theorem
I-55 of \cite{VillaniAnalysis}.
Since $\gamma$ vanishes outside $\graph (f) := \{(x,f(x)) \mid x \in X_{}\}$, there is an increasing sequence
of compact sets $K_i \subset K_{i+1} \subset \graph(f)$ whose
union $K_\infty = \lim_{i \to \infty} K_i$ contains the full mass of $\gamma$%
. Compactness of $K_i \subset \graph (f)$ implies continuity of $%
f$ on the compact projection $X_i := \pi^X(K_i)$. Thus the restriction $%
f_\infty$ of $f$ to $X_\infty := \pi^X(K_\infty)$ is a Borel map whose graph
$K_\infty = \graph(f_\infty)$ is a $\sigma$-compact set of full
measure for $\gamma$. We now verify that $\gamma$ and $(id_{X_\infty} \times
f_\infty)_\# \mu_{}$ assign the same mass to each Borel rectangle $U \times V
\subset X_{} \times Y_{}$. Since $(U \times V) \cap \mathop{\rm Graph}%
(f_\infty) = ((U \cap f_\infty^{-1}(V)) \times Y_{}) \cap \mathop{\rm Graph}%
(f_\infty)$ we find
\begin{eqnarray*}
\gamma(U \times V) &=& \gamma((U \cap f_\infty^{-1}(V)) \times Y_{}) \\
&=& \mu_{}(U \cap f_\infty^{-1}(V)),
\end{eqnarray*}
proving $\gamma = (id_{X_\infty} \times f_\infty)_\# \mu_{}$. Taking $U= X_{}
\setminus X_\infty$ and $V = Y_{}$ shows $X_{} \setminus X_\infty$ is $\mu_{}$%
-negligible. Since $id_{X_{}} \times f$ differs from the Borel map $%
id_{X_\infty} \times f_\infty$ only on the $\mu_{}$-negligible complement of
the $\sigma$-compact set $X_\infty$, we conclude $f$ is $\mu_{}$-measurable
and $\gamma = (id_{X_{}} \times f)_\# \mu_{}$ as desired.
\end{proof}

The preceding lemma shows that any measure concentrated on a graph is
uniquely determined by its marginals; $\gamma$ is therefore extremal in
$\Gamma(\pi^{X_{}}_\#\gamma,\pi^{Y_{}}_\#\gamma)$.  As the results of the
next section show,  the converse is far from being true.

\section{Numbered limb systems and extremality}

In this section we adapt Hestir and Williams \cite{HestirWilliams95}
notion of a {\em numbered limb system}
--- also called an {\em axial forest} or a {\em limb numbering system} ---
to $X \times Y$.  Using the axiom of choice,
Hestir and Williams deduced from the acyclicity condition of Bene\v{s} and \v{S}t\v{e}p\'an
\cite{BenesStepan87} that each extremal doubly stochastic measure vanishes
outside some numbered limb system.  Conversely,  they showed that vanishing
outside a numbered limb system is sufficient to guarantee extremality of a
doubly stochastic measure,  provided the graphs (and antigraphs) comprising
the system are Borel subsets of the square.
Our main theorem gives a new proof of this converse in the more general setting
of subsets $X \times Y$ of complete separable metric spaces, and under a slightly
weaker measurability hypothesis on the graphs and antigraphs.
A simple example shows that some measurability hypothesis is nevertheless
required. In the next and final section, we shall see how this converse 
relates to the question of uniqueness in optimal transportation.

Given a map $f:D\longrightarrow Y$ on $D\subset X$, we denote its graph,
domain, range, and the graph of its (multivalued) inverse by
\begin{eqnarray*}
\mathop{\rm Graph}(f):= &\{(x,f(x))\mid x\in D\},& \\
\mathop{\rm Dom}f:= &\pi ^{X}(\mathop{\rm Graph}(f))&=D, \\
\mathop{\rm Ran}f:= &&\phantom{=}\pi ^{Y}(\mathop{\rm Graph}(f)), \\
\mathop{\rm Antigraph}(f):= &\{(f(x),x)\mid x\in \mathop{\rm Dom}f  \}&\subset
Y\times X.
\end{eqnarray*}%
More typically, we will be interested in the $\mathop{\rm Antigraph}%
(g)\subset X\times Y$ of a map $g:D\subset Y\longrightarrow X$.

\begin{figure}[h]
\psfragscanon
\centering
\psfrag{o}{$o$}
\epsfig{file=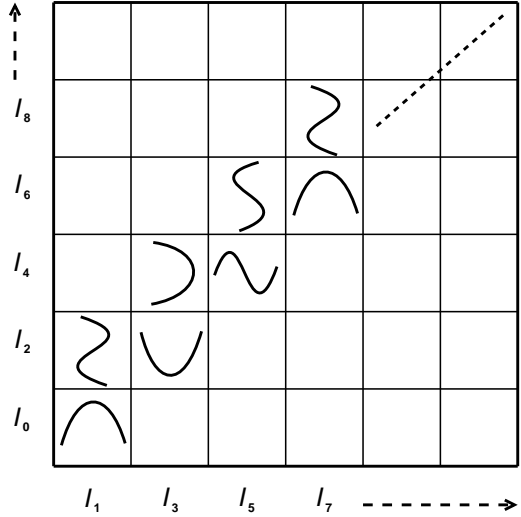, height=7cm}
\caption[Numbered limb system]
{\label{fig.numbered_limb}  \centering The subsets $I_{k}$ need not be connected; in this
numbered limb system they are represented as connected sets for visual convenience only.}
\end{figure}

\begin{definition}[Numbered limb system]
\label{numbered limb system} Let $X_{}$ and $Y_{}$ be Borel subsets of
complete separable metric spaces. A relation $S \subset X_{} \times Y_{}$ is a
\emph{numbered limb system} if there is a countable disjoint decomposition
of $X_{} = \cup_{i=0}^\infty I_{2i+1}$ and of $Y_{} = \cup_{i=0}^\infty I_{2i}$
with a sequence of maps $f_{2i}: \mathop{\rm Dom}(f_{2i}) \subset Y_{}
\longrightarrow X$ and $f_{2i+1}: \mathop{\rm Dom}(f_{2i+1}) \subset X
\longrightarrow Y$ such that $S = \cup_{i=1}^\infty \mathop{\rm Graph}%
(f_{2i-1}) \cup \mathop{\rm Antigraph}(f_{2i})$, with
$\mathop{\rm Dom}(f_{k}) \cup \mathop{\rm Ran}(f_{k+1}) \subset I_{k}$ for each
$k \ge 0$.
The system has (at most) $N$ limbs if $\mathop{\rm Dom}(f_k) = \emptyset$
for all $k>N$. 
\end{definition}

Notice the map $f_0$ is irrelevant to this definition though $I_0$ is not;
we may always take $%
\mathop{\rm Dom}(f_0) = \emptyset$, but require $\mathop{\rm Ran}(f_1)
\subset I_0$. The point is the following theorem and its corollary, which
extends and relaxes the result proved by Hestir and Williams for Lebesgue
measure $\mu_{}=\nu_{} = \lambda$ on the interval $X_{}=Y_{}=[0,1]$.
In it, $\Gamma(\mu,\nu)$ denotes the set of non-negative Borel measures
on $X \times Y$ having $\mu = \pi^X_\# \gamma$ and $\nu = \pi^Y_\# \gamma$
for marginals.
As in the preceding lemma, we say $\gamma$ {\em vanishes}
outside of $S \subset X \times Y$ if $\gamma$ assigns zero outer measure
to the complement of $S$ in $X \times Y$.

\begin{theorem}[Numbered limb systems yield unique correlations]
\label{HestirWilliams}\ \\Let $X_{}$ and $Y_{}$ be subsets of complete
separable metric spaces, equipped with $\sigma$-finite Borel measures
$\mu$ on $X$ and $\nu$ on $Y$.
Suppose there is a numbered limb system $S = \cup_{i=1}^\infty
\mathop{\rm Graph}(f_{2i-1}) \cup \mathop{\rm Antigraph}(f_{2i})$
with the property that $\graph(f_{2i-1})$ and $\antigraph(f_{2i})$
are $\gamma$-measurable subsets of $X \times Y$ for each $i \ge 1$
and for every $\gamma \in \Gamma(\mu,\nu)$ vanishing outside of $S$.
If the system has finitely many limbs or $%
\mu[X]<\infty$, then at most one $\gamma \in \Gamma(\mu,\nu)$
vanishes outside of $S$. If such a measure exists, it is given by
$\gamma = \sum_{k=1}^\infty \gamma_k$ where
\begin{eqnarray}  \label{alternating representation}
\gamma_{2i-1} = (id_{X_{}} \times f_{2i-1})_\# \eta_{2i-1},
&& \gamma_{2i} = (f_{2i} \times id_{Y_{}})_\# \eta_{2i}, \\
 \eta_{2i-1} = \Big(\mu - \pi^{X}_\#\gamma_{2i}\big)   \Big|_{\dom f_{2i-1}} ,
&& \eta_{2i} = \Big(\nu - \pi^{Y}_\#\gamma_{2i+1} \Big) \Big|_{\dom f_{2i}}.
\label{alternating marginals}
\end{eqnarray}
Here $f_k$ is measurable with respect to the $\eta_{k}$ completion of the
Borel $\sigma$-algebra. If the system has $N<\infty$ limbs, $\gamma_k=0$ for
$k > N$, and $\eta_k$ and $\gamma_k$ can be computed recursively from the
formulae above starting from $k=N$.
\end{theorem}

\begin{proof}
Let $S=\cup _{i=1}^{\infty }\mathop{\rm Graph}(f_{2i-1})\cup
\mathop{\rm
Antigraph}(f_{2i})$ be a numbered limb system whose complement has
zero outer measure for some $\sigma$-finite measure $0 \le \gamma \in \Gamma(\mu,\nu)$.
This means that $I_{k} \supset \mathop{\rm Dom} f_{k}$
gives a disjoint decomposition of $X_{}=\cup _{i=0}^{\infty }I_{2i+1}$
and of $Y_{}=\cup _{i=0}^{\infty }I_{2i}$, and that $\mathop{\rm Ran}%
(f_{k})\subset I_{k-1}$ for each $k\geq 1$.
Assume moreover,  that $\graph(f_{2i})$ and $\antigraph(f_{2i-1})$
are $\gamma$-measurable for each $i \ge 1$.  We wish to show $\gamma$ is
uniquely determined by $\mu$, $\nu$ and $S$.

The graphs $%
\mathop{\rm Graph}(f_{2i-1})$ are disjoint since their domains $I_{2i-1}$ are
disjoint, and the antigraphs $\mathop{\rm Antigraph}(f_{2i})$ are disjoint
since their domains $I_{2i}$ are. Moreover, $\mathop{\rm Graph}(f_{2i-1})$
is disjoint from $\mathop{\rm Antigraph}(f_{2j})$ for all $i,j\geq 1$: $%
\mathop{\rm Ran}(f_{2i-1})\subset I_{2i-2}$ prevents
$\mathop{\rm Graph}(f_{2i-1})$ from intersecting
$\mathop{\rm Antigraph}(f_{2j-2})$ unless $%
j=i$ since the domains $I_{2j-2}$ are disjoint, and
$\mathop{\rm Graph}(f_{2i-1})$ cannot intersect $\mathop{\rm Antigraph}(f_{2i-2})$ since $%
\mathop{\rm Dom}(f_{2i-1}) \subset I_{2i-1}$ is disjoint from $\mathop{\rm Ran}%
(f_{2i-2})\subset I_{2i-3}$.

Let $\gamma_k$ denote the restriction of $\gamma$ to $\antigraph(f_k)$
for $k$ even and to $\graph(f_k)$ for $k$ odd.  Then $\gamma = \sum \gamma_{k}$
by our measurability hypothesis,  and $\gamma_k$ restricts to a Borel
measure on $X \times \dom f_k$ if $k$ is even,  and on $\dom f_k \times Y$
if $k$ odd.
Defining the marginal projections $\mu _{k}=\pi _{\#}^{X_{}}\gamma _{k}$ and $%
\nu _{k}=\pi _{\#}^{Y_{}}\gamma _{k}$, setting
$\eta _{k}=\nu _{k}$ if $k$ even and $\eta_{k}=\mu _{k}$ if $k$ odd
yields (\ref{alternating representation}) and the
$\eta _{k}$-measurability of $f_{k}$ immediately from Lemma \ref{pure implies
unique}.
Since $\nu _{2i}$ vanishes outside $\mathop{\rm Dom}f_{2i}$, from
$\nu = \sum_{k=1}^\infty \nu_k$
we derive $\nu _{2i}=(\nu -\sum_{k \neq 2i}\nu _{k})|_{\mathop{\rm Dom}f_{2i}}$. For $k$
even, $\nu _{k}$ vanishes outside $\mathop{\rm Dom}f_{k} \subset I_k$, while for $k$
odd, $\nu _{k}$ vanishes outside $\mathop{\rm Ran}f_{k}\subset I_{k-1}$,
which is disjoint from $\mathop{\rm Dom}f_{2i}$ unless $%
k=2i+1$. Thus $\eta _{2i}=(\nu -\nu _{2i+1})|_{\mathop{\rm Dom}f_{2i}}$. The
formula (\ref{alternating marginals}) for $\eta _{2i-1}$ follows from
similar considerations.

It remains to show the representation (\ref{alternating representation})--(%
\ref{alternating marginals}) specifies $(\gamma _{k},\eta _{k})$ uniquely
for all $k\geq 1$, and hence determines $\gamma =\sum \gamma _{k}$ uniquely.
If the system has $N<\infty $ limbs, $I_{k}=\emptyset $ for $k>N$ and hence $%
\gamma _{k}=0$. We can compute $\eta _{k}$ and $\gamma _{k}$ starting with $%
k=N$, and then recursively from the formulae above for $k=N-1,N-2,\ldots ,1$%
, so the formulae represent $\gamma $ uniquely. If instead $S$ has countably
many limbs, suppose there are two finite Borel measures $\gamma $ and $\bar{%
\gamma}$ vanishing outside of $S$ and having the same marginals $\mu $ and $\nu $.
For each $k \ge 1$,  recall that
\begin{equation*}
K _{k}:=\left\{
\begin{array}{lc}
\graph(f_k) & k\ \mathrm{odd,} \\
\antigraph(f_k) & k\ \mathrm{even,}%
\end{array}%
\right.
\end{equation*}%
is measurable with respect to both $\gamma$ and $\bar \gamma$.
Given $\epsilon >0$, take $N$ large enough so that both $\gamma $ and $\bar{\gamma}$
assign mass less than $\epsilon $ to $\cup _{k=N}^{\infty }K_{k}$.
Set $\gamma_k = \gamma|_{K_k}$ and $\bar{\gamma}_{k}=\bar{\gamma}|_{K_{k}}$
and denote their marginals by $(\mu_k,\nu_k) = (\pi^X_\# \gamma_k,\pi^Y_\# \gamma_k)$
and $(\bar \mu_k,\bar \nu_k) = (\pi^X_\# \bar \gamma_k,\pi^Y_\# \bar \gamma_k)$.
Observe that both
$\gamma ^{\epsilon}:=\sum_{k=1}^{N}\gamma _{k}$ and
$\bar{\gamma}^{\epsilon }:=\sum_{k=1}^{N}\bar{\gamma}_{k}$
are concentrated on the same numbered limb system; it has
finitely many limbs, and the differences
$\delta \mu^{\epsilon }= \sum_{k=1}^N (\bar \mu_k - \mu_k)$ and
$\delta \nu^{\epsilon }= \sum_{k=1}^N (\bar \nu_k - \nu_k)$ between the marginals of
$\gamma^\epsilon$ 
and $\bar \gamma^\epsilon$ 
have total variation at most $2\epsilon $. Since the
$\delta \mu _{2i-1} = \bar \mu_{2i-1} - \mu_{2i-1}$ are mutually singular,
as are the $\delta \nu _{2i} = \bar \nu_{2i} - \nu_{2i}$,
we find the sum of the total variations of
\begin{equation*}
\delta \eta _{k}:=\left\{
\begin{array}{cc}
\bar{\mu}_{k}-\mu _{k} & k\ \mathrm{odd,} \\
\bar{\nu}_{k}-\nu _{k} & k\ \mathrm{even,}%
\end{array}%
\right.
\end{equation*}%
is bounded: $\sum_{k=1}^{N}\Vert \delta \eta _{k}\Vert _{TV(\mathop{\rm Dom}%
f_{k})}<4\epsilon $. Using (\ref{alternating representation}) to derive
\begin{eqnarray*}
\Vert \bar{\gamma}_{k}-\gamma _{k}\Vert
_{TV(X_{}\times Y_{})} &=&\left\{
\begin{array}{cc}
\Vert (id_{X_{}}\times f_{k})_{\#}\delta \eta _{k}\Vert _{TV(X_{}\times
Y_{})} & k\ \mathrm{odd,} \\
\Vert (f_{k}\times id_{Y_{}})_{\#}\delta \eta _{k}\Vert _{TV(X_{}\times
Y_{})} & k\ \mathrm{even,}%
\end{array}%
\right. \\
&=&\Vert \delta \eta _{k}\Vert _{TV(\mathop{\rm Dom}f_{k})}
\end{eqnarray*}%
%
%
%
%
%
%
%
%
%
%
%
%
%
%
%
%
and summing on $k$ yields
$\Vert \bar{\gamma}^{\epsilon }-\gamma ^{\epsilon }\Vert _{TV(X_{}\times
Y_{})}<4\epsilon $. Since $\gamma ^{\epsilon }\rightarrow \gamma $ and $%
\bar{\gamma}^{\epsilon }\rightarrow \bar{\gamma}$ as $\epsilon \rightarrow 0$%
, we conclude $\bar{\gamma}=\gamma $ to complete the uniqueness proof.
\end{proof}

As in Hestir and Williams \cite{HestirWilliams95},  the uniqueness theorem
above implies extremality as an immediate consequence.

\begin{corollary}[Sufficient condition for extremality]
\label{C:support characterization} Let $X$ and $Y$ be subsets of complete separable
metric spaces,  equipped with $\sigma$-finite Borel measures $\mu$ on $X$ and $\nu$ on $Y$.
Suppose there is a numbered limb system $S = \cup_{i=1}^\infty
\mathop{\rm Graph}(f_{2i-1}) \cup \mathop{\rm Antigraph}(f_{2i})$
with the property that $\graph(f_{2i-1})$ and $\antigraph(f_{2i})$
are $\gamma$-measurable subsets of $X \times Y$ for each $i \ge 1$,
for every $\gamma \in \Gamma(\mu,\nu)$ vanishing outside of $S$.
If the system has finitely many limbs or $%
\mu[X]<\infty$, then any measure $\gamma \in \Gamma(\mu,\nu)$
vanishing outside of $S$ is extremal in the convex set $\Gamma(\mu,\nu)$.
\end{corollary}

\begin{proof}
Suppose a measure $\gamma \in \Gamma(\mu,\nu)$ vanishes outside a numbered limb
system $S$ satisfying the hypotheses of the corollary.
If $\gamma =(1-t)\gamma
_{0}+t\gamma _{1}$ with $\gamma _{0},\gamma _{1}\in \Gamma (\mu ,\nu )$ and $%
0<t<1$, then $\gamma \geq \gamma _{0}$ and $\gamma \geq \gamma _{1}$, so both
$\gamma _{0}$ and $\gamma _{1}$ vanish outside of $S$.
According to Theorem \ref{HestirWilliams}, they are uniquely determined by
$S $ and their marginals, hence $\gamma _{0}=\gamma _{1}$ to establish the corollary.
\end{proof}

The following example confirms that a measurability gap still remains between
the necessary and sufficient conditions for extremality.
It is a close variation on the standard example of a non-Lebesgue measurable
set from real analysis.  Together with the lemma and theorem preceding,  this
example makes clear that measurability is required only to allow the graphs
to be separated from each other and from the antigraphs in an additive way.

\begin{example}[An acyclic set supporting non-extremal measures]\ \\
Let $\lambda$ denote Lebesgue measure and define the maps $f_0(x) = x$ and
$f_1(x) = x + \sqrt 2$ (mod 1) on the unit interval $X=Y=[0,1]$.
Notice $\graph(f_i) \subset [0,1]^2$ supports the doubly stochastic measure
$\gamma_i = (id \times f_i)_\#\lambda$ for $i=0$ and $i=1$;  (both measures
are extremal in $\Gamma(\lambda,\lambda)$ by Corollary \ref{C:support characterization}).
Irrationality of $\sqrt 2$ implies $S = \graph(f_0) \cup \graph(f_1)$ is an acyclic set,
hence can be expressed as a numbered limb system according to Hestir and
Williams \cite{HestirWilliams95}. On the other hand,  there are doubly stochastic
measures such as $\gamma:= \frac{1}{2}(\gamma_0+\gamma_1)$ which vanish outside of $S$ but
which are manifestly not extremal.
\end{example}

\section{Uniqueness of optimal transportation}
\label{S:unique}

In this section we apply the foregoing results
to the uniqueness question for optimal transportation on manifolds,
which arises when one wants to use a continuum
of sources to supply a continuum of sinks (modeled by $\mu$ and $\nu$ respectively)
as efficiently as possible.

Given subsets $X$ and $Y$ of complete separable metric spaces equipped with Borel
probability measures,  representing the distributions  $\mu$
of production on $X$ and $\nu$ of consumption on $Y$,  the
Kantorovich-Koopmans \cite{Kantorovich42} \cite{Koopmans49}
transportation problem is to find $\bar \gamma \in\Gamma(\mu,\nu)$ correlating
production with consumption so as to minimize the expected transportation cost
\begin{equation}\label{MKPa}
\inf_{\gamma \in \Gamma(\mu,\nu)} \int_{X \times Y} c(x,y) d\gamma(x,y)
\end{equation}
against some continuous function $c \in C(X \times Y)$.  Hereafter we shall be solely
concerned with the case in which $X$ is a differentiable manifold,  $\mu$ is
absolutely continuous with respect to coordinates on $X$,  and the cost function
$c \in C^1(X \times Y)$ is differentiable with local control on the magnitude of its
$x$-derivative $d_x c(x,y)$ uniformly in $y$;  for convenience we also suppose $Y$ to be a differentiable
manifold and $c$ is bounded,  though this is not really necessary:
substantially weaker assumptions also suffice \cite{ChiapporiMcCannNesheim10};
c.f.~\cite{GangboMcCann96} \cite{Gigli09p} \cite{FigalliGigli10p}.

In this setting one immediately asks whether the infimum (\ref{MKPa}) is
uniquely attained.  Since attainment is evident,  the question here is uniqueness.
If $c$ satisfies a {\em twist} condition, meaning
$x \in X \longmapsto c(x,y_1) - c(x,y_2)$ has no critical points for
$y_1 \ne y_2 \in Y$,  then we shall see that not only is the minimizing $\gamma$ unique,  but its
mass concentrates entirely on the graph of a single map $f_1:X \longrightarrow Y$
(a numbered limb system with one limb), thus solving a form of the transportation
problem posed earlier by Monge \cite{Monge81} \cite{Kantorovich48}.
This was proved in comparable generality by Gangbo \cite{Gangbo95} and Levin \cite{Levin99}
(see also Ma, Trudinger and Wang \cite{MaTrudingerWang05}),  building on the more specific
examples of strictly convex cost functions $c(x,y) = h(x-y)$ in $X=Y=\R^n$
analyzed by Caffarelli \cite{Caffarelli96},
Gangbo and McCann \cite{GangboMcCann95}
\cite{GangboMcCann96}, R\"uschendorf \cite{Ruschendorf95} \cite{Ruschendorf96}
and in case $h(x)=|x|^2$ by Abdellaoui and Heinich \cite{AbdellaouiHeinich94},
Brenier \cite{Brenier87} \cite{Brenier91},
Cuesta-Albertos, Matran, and Tuero-Diaz \cite{CuestaMatran89} \cite{Cuesta-AlbertosTuero-Diaz93},
Cullen and Purser \cite {CullenPurser84}  \cite{CullenPurser89} \cite{PurserCullen87},
Knott and Smith \cite{KnottSmith84} \cite{SmithKnott87},
and R\"uschendorf and Rachev \cite{RuschendorfRachev90}.
Adding further restrictions beyond this twist hypothesis allowed
Ma, Trudinger, Wang \cite{MaTrudingerWang05} \cite{TrudingerWang09b},
and later Loeper \cite{Loeper09}, to develop
a regularity theory for the map $f_1:X \longrightarrow Y$,
embracing Delano\"e \cite{Delanoe91}, Caffarelli \cite{Caffarelli92} \cite{Caffarelli96b}
and Urbas' \cite{Urbas97} results for the quadratic cost,
Gangbo and McCann's for its restriction to convex surfaces \cite{GangboMcCann00}, and
Wang's for reflector antenna design \cite{Wang96},  which involves the restriction of
$c(x,y)= -\log |x-y|$ to the sphere \cite{GlimmOliker03} \cite{Wang04}.
Unfortunately,  the twist hypothesis,  also known as a generalized Spence-Mirrlees
condition in the economic literature,
cannot be satisfied for smooth costs $c$ on compact manifolds $X \times Y$,
and apart from the result we are about to discuss there are
no general theorems which
guarantee uniqueness of minimizer to (\ref{MKPa}) in this context.  With this in mind,
let us state our main theorem,  a version of which was established
in a more complicated economic setting by Chiappori, Nesheim, and
McCann \cite{ChiapporiMcCannNesheim10}.  The streamlined formulation
and argument given below should prove more interesting and accessible to
a mathematical readership.


\begin{theorem}[Uniqueness of optimal transport on manifolds]\label{T:unique}
\ \\Let $X$ and $Y$ be complete separable manifolds equipped with Borel probability
measures $\mu$ on $X$ and $\nu$ on $Y$.  Let $c \in C^1(X \times Y)$ be a bounded
cost function such that for each $y_1 \ne y_2 \in Y$, the map
\begin{equation}\label{subtwist}
x \in X \longmapsto c(x,y_1) - c(x,y_2)
\end{equation}
has no critical points,  save at most one global minimum and at most one global maximum.
Assume $d_x c(x,y)$ is locally bounded in $x$,  uniformly in $Y$.
If $\mu$ is absolutely continuous in each coordinate chart on $X$,
then the minimum (\ref{MKPa}) is uniquely attained;  moreover,  the minimizer
$\gamma \in \Gamma(\mu,\nu)$ vanishes outside a numbered limb system having at most two
limbs.
\end{theorem}

\begin{proof}
We first prove that there is a numbered limb system having at most
two limbs,  outside of which the mass of all minimizers $\gamma$ vanishes.
A detailed argument confirming the plausible fact that the graphs of these limbs are
Borel subsets of $X \times Y$ will be given later.
Uniqueness of $\gamma$ then follows from Theorem \ref{HestirWilliams}.

By linear programming duality (due to Kantorovich and Koopmans in this context),
it is well-known \cite{Villani09}
that there exist upper semi-continuous potentials $q \in L^1(X,d\mu)$ and $r \in L^1(Y,d\nu)$ with
\begin{equation}\label{c-transform}
q(x) = \inf_{y \in Y} c(x,y) - r(y)
\end{equation}
such that
\begin{equation}\label{duality}
\inf_{\gamma \in \Gamma(\mu,\nu)} \int_{X \times Y} c(x,y) d\gamma(x,y) = \int_X q(x) d\mu(x) + \int_Y r(y) d\nu(y).
\end{equation}
From (\ref{c-transform}) we see
\begin{equation}\label{zero set}
c(x,y) - q(x) - r(y) \ge 0,
\end{equation}
and let
\begin{equation}
Z:=\{(x,y)\in X\times Y | c(x,y)-q(x)-r(y)=0\}
\end{equation}
denote the set where the non-negative function $c(x,y)-q(x)-r(y)$ vanishes. Lower semi-continuity of this function implies $Z$ is a  closed subset of $X\times Y$.
Notice that  (\ref{duality}) implies any minimizer $\gamma \in \Gamma(\mu,\nu)$
vanishes outside the zero set $Z \subset X \times Y$ of the non-negative
function appearing in (\ref{zero set}).  It remains
to show this set $Z$ is contained in a numbered limb system consisting of
at most two limbs (apart from a $\mu \otimes \nu$ negligible set).

From (\ref{c-transform}), $q$ is locally Lipschitz, since
$d_x c(x,y)$ is controlled locally in $x$, independently of $y \in Y$.
Rademacher's theorem therefore combines with absolute continuity of
$\mu$ to imply $q$ is differentiable $\mu$-almost everywhere;  we can
safely ignore any points in $X$ where differentiability of $q$ fails,
since they constitute a set of zero volume: $\gamma[\dom Dq \times Y]
= \mu[\dom Dq]=1$.
Taking $x_0 \in \dom Dq$,  suppose $(x_0,y_1)$ and $(x_0,y_2)$ both lie in $Z$,
hence saturate the inequality (\ref{zero set}).
Then $d_x c(x_0,y_1) = Dq(x_0) = d_x c(x_0,y_2)$.
In case the cost is twisted,  meaning
(\ref{subtwist}) has no critical points,  we conclude $y_1=y_2$ hence
$Z \cap (\dom Dq \times Y)$ is contained in a graph.
This completes the proofs by Gangbo and Levin
of existence (and uniqueness) of a solution $y_1=f_1(x_0)$ to Monge's problem,
pairing almost every $x_0 \in X$ with a single $y_1 \in Y$.
Notice uniqueness follows from Lemma \ref{pure implies unique} without further
measurability assumptions.

In the present setting,  however,  we only know that $x_0$ must be
a global minimum or global maximum of the function (\ref{subtwist}).  Exchanging
$y_1$ with $y_2$ if necessary yields
\begin{equation}\label{keep away}
q(x) \le c(x,y_1) - r(y_1)  \le c(x,y_2) - r(y_2)
\end{equation}
for all $x \in X$,  the second inequality being strict unless $x=x_0$,
in which case both inequalities are saturated.  Strictness of inequality
(\ref{keep away}) implies $(x,y_2) \not\in Z$ unless $x=x_0$.  In other words,
$(x,y_2) \in Z$ lies on the antigraph of a function $f_2(y_2) = x_0$
well-defined at $y_2$.  There may or may not be a point
$y_0 \in Y$ different from $y_1$ such that
\begin{equation}
\label{Graph1}
q(x) \le c(x,y_0) - r(y_0) \le c(x,y_1) - r(y_1)
\end{equation}
for all $x \in X$.  If such a point $y_0$ exists,  then
$(x_0,y_1) \in \antigraph(f_2)$ as above.  If no such $y_0$ exists,
setting $f_1(x_0) := y_1$ yields
$Z \cap (\dom Dq \times Y) \subset \graph(f_1) \cup \antigraph(f_2)$.
Since the range of $f_1$ is disjoint from the domain of $f_2$,  this completes
the proof that --- up to $\gamma$-negligible sets ---
$Z$ lies in a numbered limb system with at most two limbs, as desired.

Let us now prove Borel measurability of these {\it limbs}. To do this, we define the cross-difference
as in McCann \cite{McCann99},
$$ \Delta(x,y,x',y'):= c(x,y)+c(x',y')-c(x,y')-c(x',y)$$
which is a continuous function on $(X\times Y)^2$ and notice that $\Delta \leq 0$ on $Z^2$, i.e, any two  $(x,y)$ and $(x',y')$ in $Z$ satisfy
$$\Delta(x,y,x',y') \leq 0 $$
This well-known fact \cite{SmithKnott92} can be deduced by summing the inequalities
\begin{eqnarray*}
0 &\leq  c(x',y) - q(x')- r(y) = c(x',y)-q(x') +q(x) - c(x,y)\\
0 &\leq   c(x,y') - q(x)- r(y') = c(x,y')-q(x) +q(x') - c(x',y').
\end{eqnarray*}

Closedness of $Z$ and $\sigma$-compactness of $B=X\times Y$ imply
 $$h(x_1,y_1):= \inf_{\{ (x,y_2)\in X\times Y |(x_1,y_2)\in Z  \}} \Delta(x_1,y_1,x,y_2) $$
 is Borel on $X\times Y$, according to Lemma \ref{Borel} below.
 Taking $y_2=y_1 $ implies $h \le 0$ on $Z$.
 A point $(x_1,y_1)\in Z $ is said to be {\it marked} if $x_1\in X$ and $h(x_1,y_1)=0$, i.e
 \begin{equation}
 \label{marked}
  c(x,y_1) - c(x,y_2)\le c(x_1,y_1)-c(x_1,y_2)
  \end{equation}
 for all $x\in X$ and $(x_1,y_2)\in Z$. This definition is equivalent to saying that there is no
 $y_0$ satisfying (\ref{Graph1}) , i.e , the set of marked points in
 $Z\cap ({\rm Dom} Dq \times Y )$ is equal to Graph($f_1$). This implies Graph($f_1$) =
 $({\rm Dom} Dq \times Y )\cap Z\cap \{(x,y)|h(x,y)=0\} $ hence is a Borel subset of $X\times Y$.
 Borel measurability of $\{h>0\}$ and hence Antigraph($f_2$) also follows.
 \end{proof}


\begin{lemma}
\label{Borel}
Let A and B be topological spaces and $ Z\subset A\times B$ be closed. If
$g:A\times B\rightarrow {\bf R}\cup\{-\infty\}$ is lower semi-continuous,
and B is $\sigma$-compact, then
the following function $h$ is Borel:
$$ h(a):=\inf_{\{b\in B|(a,b)\in Z\}} g(a,b).$$
\end{lemma}

\begin{proof}
See Lemma A.4 of  \cite{ChiapporiMcCannNesheim10}.
\end{proof}

Let us conclude by recalling an example of an extremal doubly stochastic measure
which does not lie on the graph of a single map,  drawn from
work of Gangbo and McCann \cite{GangboMcCann00} and Ahmad \cite{Ahmad04}
on optimal transportation,
and developed in an economic context by Chiappori, McCann, and
Nesheim~\cite{ChiapporiMcCannNesheim10}.
Other examples may be found in the work of Seethoff and Shiflett~\cite{SeethoffShiflett78},
Losert~\cite{Losert82}, Hestir and Williams~\cite{HestirWilliams95},
Gangbo and McCann \cite{GangboMcCann96}, Uckelmann \cite{Uckelmann97},
McCann \cite{McCann99}, and Plakhov \cite{Plakhov04a}.

Imagine the periodic interval $X=Y = \R / 2\pi \Z=[0,2\pi[$ to parameterize a town built on
the boundary of a circular lake,  and let probability measures $\mu$ and $\nu$ represent
the distribution of students and available places in schools, respectively.
Suppose the distribution of students is smooth and non-vanishing but peaks sharply at the
northern end of the lake,  and the distribution of schools is smooth, non-vanishing
and sharply peaked at the southern end of the lake.
If the cost of transporting a student residing at location $\theta \in [0,2\pi]$
to school at location $\phi \in [0,2\pi]$ is presumed to be given in terms of the
angle commuted by $c(\theta,\phi) = 1 - \cos(\theta-\phi)$,  the most effective pairing of
students with places in schools is given by the measure in $\Gamma(\mu,\nu)$
which attains the minimum:
\begin{equation}\label{MKP}
\min_{\gamma \in \Gamma(\mu,\nu)} \int_{X \times Y} c(\theta,\phi) \; d\gamma(\theta,\phi).
\end{equation}

According to results of Gangbo and McCann \cite{GangboMcCann00}
which are generalized in Theorem \ref{T:unique},
this minimizer is unique, and its support is contained in the union of
the graphs of two maps ${\mathbf{t^\pm}}: X \longrightarrow Y$. A schematic
illustration is
given in Figure \ref{fig.extreme_doubly_stochastic}, where the restriction of
the support to the subsets marked by $\pm$ on the flat torus $X \times Y$ represent
$graph(\mathbf{t^+})$ and  $graph(\mathbf{t^-})$ respectively. The dotted lines mark
$\phi - \theta = \pm \frac{\pi}{2}, \pm \frac{3\pi}{2}$.
The necessary positivity of $\gamma[ J_{X} \times J_{Y1}]>0$
in this picture may be explained by observing that although
it is cost-effective for all students to attend a school where they live,  this is
incompatible with the concentration of students at the north end of the lake,  and
of schools at the south end.  Once this imbalance is corrected by sending a sufficient
number of northern students to southern schools by the map $\mathbf{t^-}$,
the remaining students can be assigned to school near their home using the map
$\mathbf{t^+}$.
Continuity of both of these maps is established in \cite{GangboMcCann00}
and further quantified by McCann and Sosio \cite{McCannSosio10p},
and McCann, Pass and Warren \cite{McCannPassWarren10p}.
Periodicity of graphs on the flat torus can be used to represent the support as a
numbered limb system in more than one way; see
Figure \ref{fig.example_numbered_limb}, which exploits the fact that the support
of $\gamma$ in Figure \ref{fig.extreme_doubly_stochastic} intersects
$X \times J_{Y2}$ in a graph and $X \times \left(Y - J_{Y1}\right)$
in an anti-graph.

Chiappori, Nesheim and McCann \cite{ChiapporiMcCannNesheim10}
called the uniqueness hypothesis limiting the number of critical points
to at most one maximum and at most one minimum in (\ref{subtwist})
the {\em subtwist} condition.
{Although it is satisfied in the example above,  it is an unfortunate fact
that the subtwist condition cannot be satisfied
by any smooth function $c(\theta,\phi)$ on a product of manifolds $X \times Y$
with more complicated Morse structures than the sphere.  It is an interesting
open problem to find a criterion on a smooth cost $c(\theta,\phi)$ on
$X=Y = \R^2/\Z^2$ which guarantees uniqueness of the minimum (\ref{MKP}) for all
smooth densities $\mu$ and $\nu$ on the torus.
Although we expect such costs to be generic,  not a single example of such a cost
is known to us.  Hestir and Williams criteria for extremality seems likely to remain
relevant to such questions, and it is natural to conjecture that the complexity of the
Morse structure of the manifold $X$ plays a role in determining the required number
of limbs in the system.}

\begin{figure}[h]
\psfragscanon
\centering
\psfrag{o}{$o$}
\epsfig{file=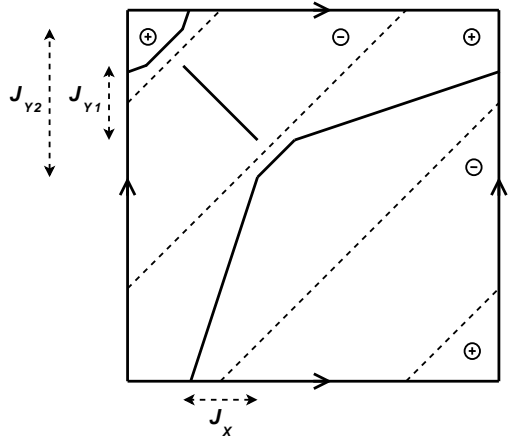, height=6cm}
\caption[An extreme doubly stochastic measure]
{\label{fig.extreme_doubly_stochastic}  \centering Schematic support of
the optimal measure from the example.}
\end{figure}

\begin{figure}[h]
\psfragscanon
\centering
\psfrag{o}{$o$}
\epsfig{file=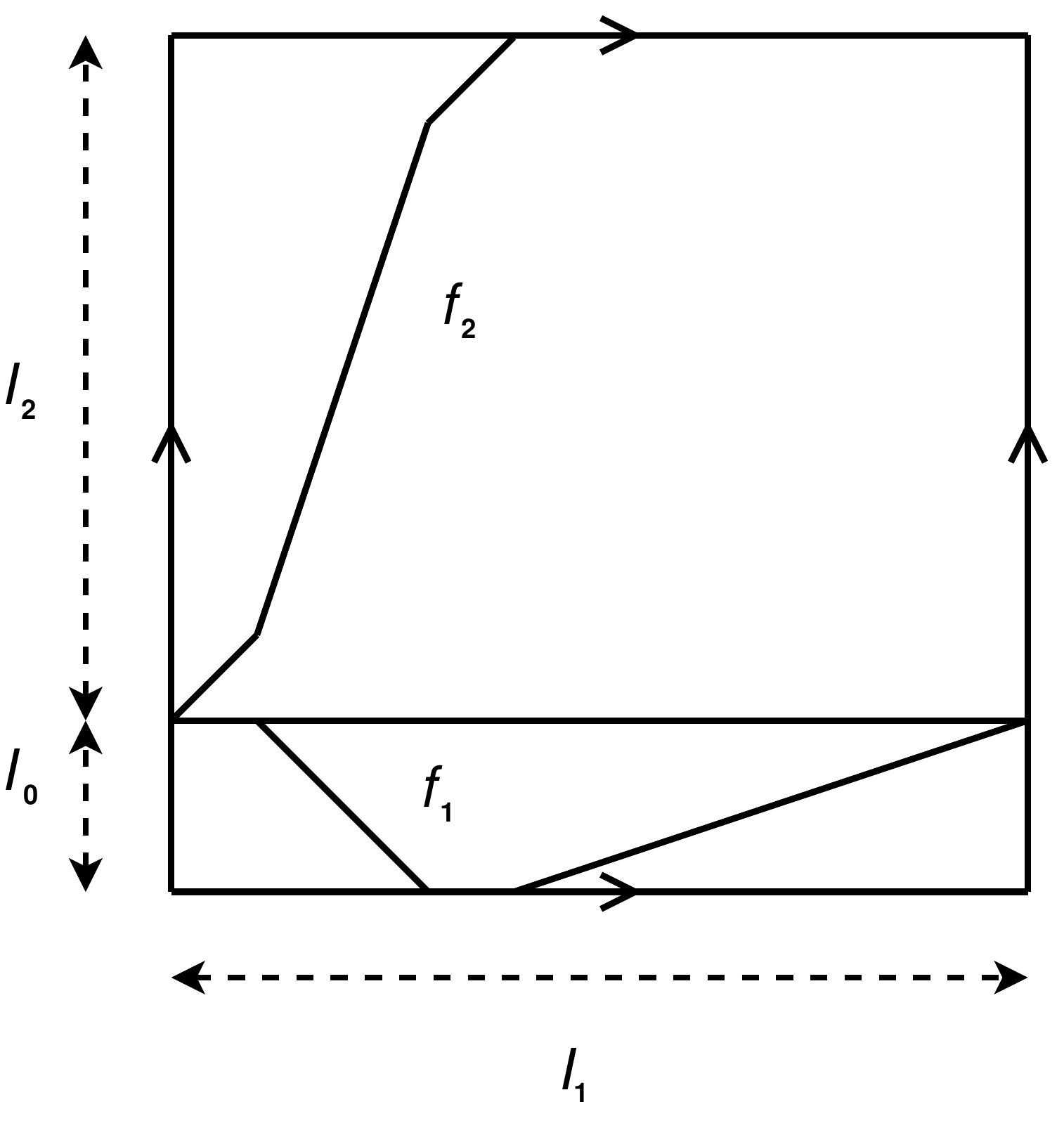, height=6cm}
\epsfig{file=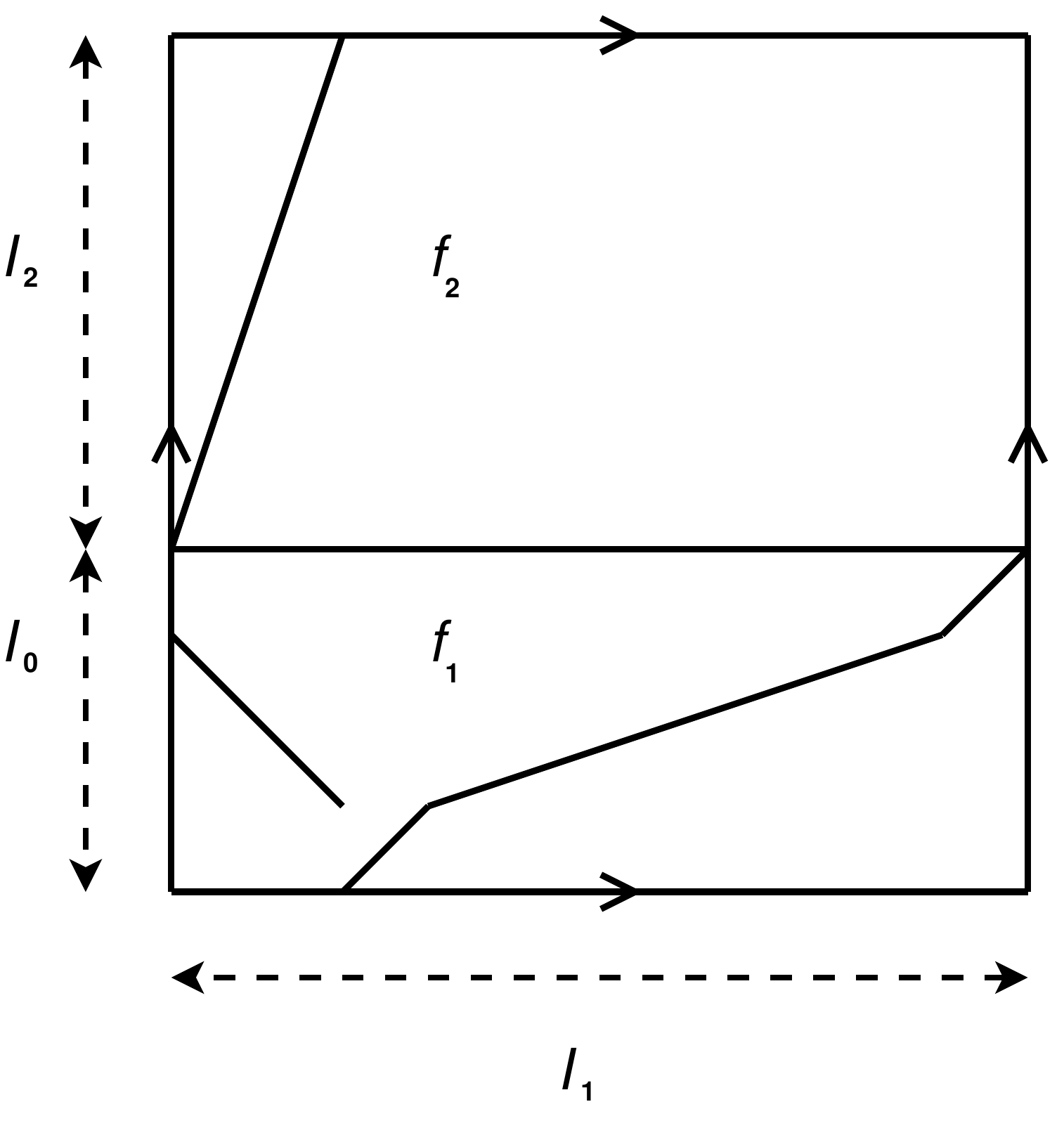, height=6cm}
\caption[Numbered limb system]
{\label{fig.example_numbered_limb}  \centering Two different numbered limb systems which represent Figure \ref{fig.extreme_doubly_stochastic}.}
\end{figure}





\section{Epilog}
The connections of optimal transportation to geometry and curvature ---
sectional \cite{KimMcCann07p} \cite{Loeper09},
Ricci  \cite{CorderoMcCannSchmuckenschlager01} \cite{Lott09} \cite{LottVillani09}
\cite{McCannTopping10} \cite{OttoVillani00} \cite{Sturm06ab} \cite{Villani09}, and
mean \cite{KimMcCannWarren09p} --- have become abundantly clear in recent years.
Connections to differential topology remained largely unsuspected.  The results
reviewed above highlight the delicacy of identifying the extremality of a doubly
stochastic measure from its support,  and the role played by critical points of the
transportation cost \eqref{subtwist} in guaranteeing the uniqueness of the
extremal measure $\gamma \in \Gamma(\mu,\nu)$ which solves a Kantorovich transportation problem
\eqref{MKP} set on the ball or sphere $X$.  When the sources $\mu$ are continuously distributed,
the topology of the landscape $X$ limits the support of
$\gamma$ to lie on a graph in the case of a ball,  and a numbered limb system
with two limbs in the case of a sphere.  This characterization is dimension independent.
For landscapes with more complicated topology,  not a single example of a cost function
$c \in C^1(X \times Y)$ is known to guarantee uniqueness of
optimal measure for all continuous densities $\mu$ and $\nu$ ---
nor is anything known about the support of $\gamma$ beyond its numbered limb system structure
and the local rectifiability determined by the rank of the cost \cite{McCannPassWarren10p}
\cite{Pass10p}.


\end{document}